\documentclass[a4paper,english]{article}

\usepackage[T1]{fontenc}
\usepackage[utf8]{inputenc}
\usepackage[english]{babel}
\usepackage{amsthm}
\usepackage{amsmath}
\usepackage{amssymb}
\usepackage{graphicx}
\usepackage{esint}
\usepackage[unicode=true,pdfusetitle, bookmarks=true,bookmarksnumbered=false,bookmarksopen=false,
 breaklinks=false,pdfborder={0 0 0},backref=false,colorlinks=false]
 {hyperref}
\hypersetup{pdftitle={d},
 pdfauthor={Sergei Kalmykov and Elena Prilepkina}}

\makeatletter
\usepackage{enumitem}		
\theoremstyle{plain}
\newtheorem{thm}{\protect\theoremname}
  \theoremstyle{plain}
  \newtheorem*{thm*}{\protect\theoremname}
\theoremstyle{plain}
\newtheorem{lem}{\protect\lemmaname}
\theoremstyle{plain}

\theoremstyle{plain}
\newtheorem{cor}{\protect\corollaryname}

\makeatother

\usepackage{babel}
\providecommand{\lemmaname}{Lemma}
\providecommand{\theoremname}{Theorem}
\providecommand{\propositionname}{Proposition}
\providecommand{\corollaryname}{Corollary}

\title{Extremal decomposition problems for p-harmonic radius}

\author{ Sergei Kalmykov, Elena Prilepkina }

\date{\today}

\begin{document}
\maketitle

\begin{abstract}
We extend  classical results by Lavrent'ev and Kufarev concerning the product of the conformal radii of planar non-overlapping domains. We also extend relatively recent results  for  the case of domains in the $n$-dimensional Euclidean space,
$n \geq 3$, with conformal radii replaced by harmonic ones. Namely, we get analogues of these results in $n$-dimensional Euclidean space in terms of $p$-harmonic radius. The proofs are based on technique of modulii of curve families and dissymmetrization of such families.


Classification (MSC 2010): 31A15, 31C45, 31B99

\end{abstract}

\section{Introduction}

Extremal decomposition problems are due to well-known Lavrentiev's inequality
\begin{equation}\label{lavr}
r(a_1,D_1)\cdot r(a_2,D_2)\leq |a_1-a_2|^2
\end{equation}
where $r(a_1,D_1)$ and $r(a_2,D_2)$ are conformal (inner) radii of planar disjoint  domains, $a_i\in D_i$, $i\in 1,2$.
  The conformal radius plays an important role in geometric function theory. Its generalization to higher dimensional  domains is known as $p$-harmonic radius introduced  by Levitski\u{i} in \cite{Lev}. For $p=2$ we deal with  harmonic radius which has various applications to partial differential equations (see, for example, \cite{Ban}). The applications of the extremal decomposition problems to analytic functions are numerous including distortion theorem, coefficient inequalities,   polynomial  inequalities and other similar problems.
Therefore it is desirable to
extend  extremal decomposition problems to higher-dimensional domains.  In this paper we are going to obtain  inequalities of a similar type as~\ref{lavr}) for the $p$-harmonic radius.

Note that, for planar domains, the inequality (\ref{lavr}) was generalized in several directions. The notion of Robin radius (the inner radius is a particular case of it) was introduced in papers by V.N. Dubinin  and his students (see \cite{Dub1} and references therein). Some extremal decomposition problems for Robin radius were considered there as well. The approach used there was essentially based on the fact that the sum of harmonic functions is again a harmonic function. In $\mathbb{R}^n$, the sum of $p$-harmonic functions is not a $p$-harmonic function in general. So it was possible only to obtain similar results for harmonic radius  in \cite{DP} and for 2-harmonic Robin radius in \cite{GKP}. In \cite{W}, W.~Wang considered the $n$-harmonic radius and extended the method of the harmonic transplantation to that of the $n$-harmonic transplantation.

G.V. Kuz'mina, A.Yu. Solynin, E.G. Emel'yanov, A. Vasiliev, Ch.~Pommerenke  used a different approach to study extremal decomposition problems, namely, their technique was based on the method of extremal metric (see for example \cite{Kuz,Emel,Vas,VP,Sol} and references therein).

In this paper, we also use the technique of moduli of curve families to prove theorems on extremal decomposition for the  $p$-harmonic radius. To formulate our results, we need some definitions and notation.

Throughout the paper  $\mathbb{R}^n$ denotes the $n$-dimensional
Euclidean space consisting of points $x=(x_1,\dots,\,x_n)$, $n\geq2$, $y=(y_1,\dots,\ y_n)\in \mathbb{R}^n$, $<x,y>=\sum_{i=1}^n x_iy_i$ is the inner product of $x$ and $y$,
and $|x|=\sqrt{x_1^2+\dots+x_n^2}$ is the length of a
vector $x\in\mathbb{R}^{n}$. For a ball and hypersphere, we introduce the following notation:
$B(a,r)=\{x\in\mathbb{R}^{n}:\,|a-x|<r\},$
$S(a,r)=\{x\in\mathbb{R}^{n}:\,|a-x|=r\}$, ~~~$a\in\mathbb{R}^{n}$, respectively.  For $\tau\in \mathbb{R}$ and $a\in \mathbb{R}^n \setminus \{0\}$,   we denote by
$L(a,\tau)=\{x\in  \mathbb{R}^{n}:<x,a>=\tau\}$ a hyperplane perpendicular to the vector $a.$
In what follows, we need the cylindrical coordinates $[\rho,\theta, x']$ of a point $x=(x_1,...,x_n) \in \mathbb{R}^n$, $n\geq 2,$ connected with the initial coordinates by the formulas: $x_1=\rho \cos \theta$, $x_2=\rho \sin \theta$, $x'=(x_3, x_4, \dots, x_n)$.

By the rotation by an angle $\beta$, we understand the transformation: $[\rho,\theta,x']\mapsto [\rho, \theta+\beta,x']$.

If $p>1$ then the {\it $p$-Laplacian} is defined as
$$
\Delta_p u=-{\rm div} (|\nabla u|^{p-2}\nabla u),
$$
on $\mathbb{R}^n$ for $u\in C^2(\mathbb{R}^n)$. For the potential theory for the $p$-Laplacian,  we refer to~\cite{HKM}
 and references therein. Let $D$ be a domain in $\mathbb{R}^n$, $x_0\in D$ ($x_0\neq \infty$), $\delta(x_0)$ be Dirac delta measure or function at the point $x_0$, $\omega_n$ be the volume of the $n$-dimensional unit ball in $\mathbb{R}^n$. It is known that, in the domain $D$ with a regular boundary, there exists a generalized solution $u_D(x,x_0)\in C^{1,\alpha} (D\setminus \{x_0\})$, $\alpha>0$, of the following Dirichlet problem
$$
\left\{ \begin{array}{l}
-\Delta_p u =n \omega_n \delta(x_0)\\
u=0 \ \ \ \text{on} \ \partial D.
\end{array}\right.
$$
The function $u_D(x,x_0)$ is called {\it $p$-harmonic Green function} of the domain $D$ with a pole at the point $x_0$.

If  we introduce the notation
$$
\mu_p(t)=\begin{cases}
-\log(t),\ p=n,\\ \frac{1}{\gamma}t^{-\gamma}, \ \text{where} \ \gamma=\frac{n-p}{p-1}, \ p\ne n,
\end{cases}
$$
for $t>0$ then we get from the results of \cite{KV} that the difference
$$
h_p(x,x_0)=u_D(x,x_0)-\mu_p(|x-x_0|)
$$
belongs to the class $L^{\infty}(D)$. The quantity $R_p(x_0,D)\geq 0$, such that

$$
\lim_{x\to x_0} h_p(x,x_0)=-\mu_p(R_p(x_0,D)),
$$
is called the {\it inner $p$-harmonic radius} of the domain $D$ at the point $x_0$   \cite{Lev,KV}. For a domain $D$ with not smooth boundary by the {\it inner $p$-harmonic radius} at a point $x_0\in D$, we will call the  quantity
$$
R_p(x_0,D) =\sup R_p(x_0,D'),
$$
where the supremum is taken over all domains with smooth boundaries and such that $D'\subset D$. In what follows, we will call the quantity $R_p(x_0,D)$ simply {\it $p$-harmonic radius} if $p\ne 2$ and {\it harmonic radius} if $p=2$.

Non-linearity of $p$-harmonic functions and the fact that  the set of conformal mappings in spaces of dimension greater than 2 is much more restricted than in the planar case make  calculations of $p$-harmonic radii of extremal domains complicated. In the case when $p=n$, we may apply M\"{o}bius transformations. It is easy to see that the mapping
$$
f(x)=-\frac{a}{|a|}+\frac{2|a|(x+a)}{|x+a|^2}
$$
maps a point $a\in \mathbb{R}^n$ to the origin and the hyperplane $L(a,0)$ onto the hypershpere $S(0,1)$. According to  \cite[formula (2.21)]{W} we get
\begin{equation}\label{rad1}
R_n(a,B^*)=R_n(0,B(0,1))/|f'(a)|=\left|a-\tilde{a}\right|,
\end{equation}
where  $|f'(a)| := | \det Df(a)|^{1/n},$ $B^*$ is the half-space containing the point $a$ and with the boundary $L(a,0)$. By $\tilde{a}$ we denote a point symmetric to $a$ with respect to the mentioned hyperplane $\partial B^*.$

Note that if a function $y = f(x)$  conformally maps a domain $B\subset \mathbb{R}^n$  onto a domain $\tilde B\subset \mathbb{R}^n$
and $y_0 = f(x_0),$ $x_0\in B,$ then we have \cite[p. 196]{Ban}
\begin{equation}\label{harm}
R_2(\tilde B,y_0) =|f'(x_0)| R_2(B,x_0).
\end{equation}
 Using the symmetry principle for harmonic
functions, it is not difficult to see \cite{DP} that the harmonic radius of  the dihedral angle $B_{2k}^*=\left\{x=[\rho,\theta,x']:|\theta|<\frac{\pi}{2k}\right\}$ at the point $x_0=[t,0,0]$ is
\begin{equation}\label{rad2}
R_2(B^*_{2k},x_0)=\left(\sum\limits_{l=1}^{2k-1}(-1)^{l-1}|x_0-x_l|^{2-n}\right)^{\frac{1}{2-n}},
\end{equation}
where $x_l=[t,\pi l/k, 0], \ l=1,\ldots, 2k-1, \ k=1,2,\ldots.$

Other special cases of calculation of p-harmonic radii that we do not use here can be found in the following papers \cite{Ban}, \cite{W}.

\section{Statements of main results}

If we consider the function $\mu_2(t)=-\log(t)$ as the fundamental solution of the Laplace equation then the inequality (\ref{lavr}) can be written in the following form
$$
\log r(a_1,D_1) + \log r(a_2,D_2) \leq 2 |a_1-a_2|,
$$
or equivalently
\begin{equation}
\mu_2(r(a_1,D_1))+\mu_2(r(a_2,D_2))\geq \mu_2(r(a_1,D_1^*))+\mu_2(r(a_2,D_2^*)),
\end{equation}
where $D_1^*$ and $D_2^*$ are half-planes with a common boundary $L^*$ such that the points $a_1$, $a_2$ are symmetric with respect to $L^*$. As a corollary of the  following theorem  we show that Lavrentiev's inequality remains true for $p$-harmonic radii of non-overlapping domains in the Eucledean space.

\begin{thm}\label{lavr1}
Let $G$ be a domain symmetric with respect to a hyperplane $L$, $a_1\in G$, $a_2\in G$, and let points $a_1$, $a_2$ $(a_1\ne a_2)$ be symmetric with respect to $L$ as well. Then for any non-overlapping domains $D_1\subset G$, $D_2\subset G$ such that $a_1\in D_1$, $a_2\in D_2$, we have
$$
\mu_p(R_p(a_1,D_1))+\mu_p(R_p(a_2,D_2))
$$
$$
\geq \mu_p(R_p(a_1,D_1^*))+\mu_p(R_p(a_2,D_2^*)),
$$
where
$D_1^*$ and $D_2^*$ are domains
obtained by division of the G   by hyperplane L, i.e.  $D_1^*$ and $D_2^*$ are non-overlapping   symmetric   to each other with respect to $L$ domains and such that $D_1^*\cup D_2^*=G\setminus L$.
\end{thm}

\begin{cor}
Let $a_1$, $a_2$ be arbitrary points in $\mathbb{R}^n$, $D_1$, $D_2$ be non-overlapping domains in $\mathbb{R}^n$, $a_i\in D_i$, $i=1,2$. Then
$$
\mu_p(R_p(a_1,D_1))+\mu_p(R_p(a_2,D_2))
$$
$$
\geq \mu_p(R_p(a_1,D_1^*))+\mu_p(R_p(a_2,D_2^*))=2\mu_p(R_p(a_1,D_1^*)),
$$
where $D_1^*$ and $D_2^*$ are half-spaces with common boundary $L^*$ such that the points $a_1$ and $a_2$ are symmetric with respect to $L^*$. In particular, for $p=n$  by~$(\ref{rad1})$ we obtain Lavrentiev's inequality
$$
R_n(a_1,D_1)\cdot R_n(a_2,D_2)\leq |a_1-a_2|^2.
$$
\end{cor}

Theorem \ref{lavr1} also allows us to extend well-known  Kufarev's inequality concerning the product of inner radii of subdomains of the unit disk \cite[Section 6]{Dub1} to the case of $n$-harmonic radius. Let $a_1$, $a_2$ be arbitrary points of the ball $B(0,1)$, $a_1\ne 0$. The inversion $y=f_1(x)=a+r^2(x-a)/|x-a|^2$ with parameters $a=a_1/|a_1|^2$, $r^2=|a|^2-1$, preserves the ball $B(0,1)$ and maps the point $a_1$ to the origin. The second inversion $z=f_2(y)=b+\rho^2(y-b)/|y-b|^2$ with parameters
$$
b=\frac{1+\sqrt{1-|f_1(a_2)|^2}}{|f_1(a_2)|^2}f_1(a_2), \ \ \rho^2=|b|^2-1,
$$
preserves $B(0,1)$ and maps the points $0$ and $f_1(a_2)$ to a pair of symmetric points with respect to the origin. Therefore, the composition
\begin{equation}\label{phi}
\psi_{a_1,a_2}(x)=f_2(f_1(x))
\end{equation}
preserves the
unit ball and maps the points $a_1$ and $a_2$ to some symmetric points $c$ and $-c$
respectively. If $a_1 = 0$ then $f_1(x) = x$. Denote by $C(a_1, a_2)$ the image of
hyperplane $< c, x >= 0$ under the mapping $\psi_{a_1,a_2}^{-1}$. It is
clear that $C(a_1, a_2)$ is a  "hypersphere" (either a hyperplane or hypersphere) which is orthogonal to
the sphere $S(0, 1)$ and lies between the points $a_1$ and $a_2$.

\begin{cor}\label{cor}
Let $a_1$, $a_2$ be arbitrary points in $B(0,1)$, $D_1$, $D_2$ be non-overlapping domains in $B(0,1)$, $a_i\in D_i$, $i=1,2$. Then
$$
R_n(a_1,D_1)R_n(a_2,D_2)\leq R_n(a_1,D_1^*)R_n(a_2,D_2^*))
$$
where $D_1^*$ and $D_2^*$ are domains
obtained by division of the ball $B(0,1)$   by "hypersphere" $C(a_1,a_2)$. \end{cor}

An analogue of Kufarev's inequality for $p=2$, $n\geq 3$ was obtained in \cite[Theorem~3]{DP}. According to this result, the quantity
$$
-R_2(a_1,D_1)^{2-n}-R_2(a_2,D_2)^{2-n}
$$
attains its maximum when $D_1$ and $D_2$ are subdomains of the unit ball that are described by the following inequalities
$$
D_l=\{x\in B(0,1):\sum_{k=1}^2 (-1)^{k+l}\left(|x-a_k|^{2-n}-||a_k|x-a_k/|a_k||^{2-n}\right)>0\}, \ \ l=1,2.
$$

Figure 1  depicts the section of the domains $D_l$, $l=1,2$,  by the plane $(x_1,x_2,0,0)$ for the case of $n=4$, $p=2$ with $a_1=(1/2,0,0,0)$ and $a_2=(1/3, 0,0,0)$. Also this figure shows the extremal decomposition from Corollary~\ref{cor} for $p=n=4$. The problem of obtaining an analogue of Kufarev's theorem for arbitrary $p>1$ is interesting and still open.
\begin{figure}
\centering
\includegraphics[width=0.6\textwidth]{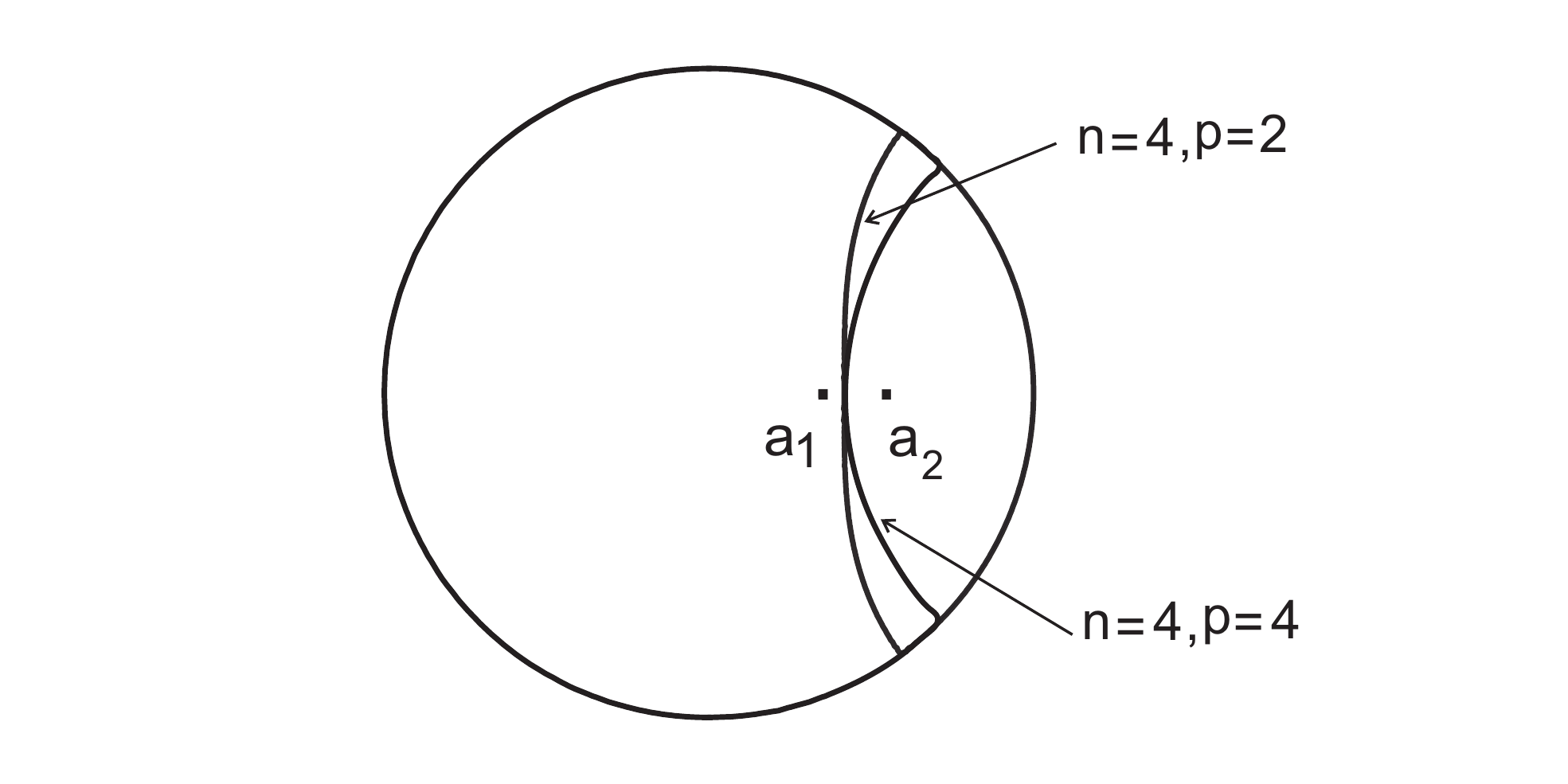}
\caption{Example of exteremal decompositions}
\end{figure}

By applying (\ref{harm}) to the harmonic radius, we conclude that the extreme configuration of Corollary \ref{cor} is preserved for the
quantity
$$\left(|\psi'_{a_1,a_2}(a_1)|R_2(a_1,D_1)\right)^{2-n}+\left(|\psi'_{a_1,a_2}(a_2)|R_2(a_2,D_2)\right)^{2-n},$$ where
$\psi_{a_1,a_2}(x)$ is defined by (\ref{phi}).
\begin{cor}\label{cor2}
Let $a_1$, $a_2$ be arbitrary points in $B(0,1)$, $D_1$, $D_2$ be non-overlapping domains in $B(0,1)$, $a_i\in D_i$, $i=1,2$. Then
$$
\left(|\psi'_{a_1,a_2}(a_1)|R_2(a_1,D_1)\right)^{2-n}+\left(|\psi'_{a_1,a_2}(a_2)|R_2(a_2,D_2)\right)^{2-n}\geq$$
$$
\left(|\psi'_{a_1,a_2}(a_1)|R_2(a_1,D_1^*)\right)^{2-n}+\left(|\psi'_{a_1,a_2}(a_2)|R_2(a_2,D_2^*)\right)^{2-n},
$$
where $D_1^*$ and $D_2^*$ are domains
obtained by division of the ball $B(0,1)$   by "hypersphere" $C(a_1,a_2)$. \end{cor}

In the following theorem we solve the problem on extremal decomposition of a ring or cylinder  with free poles belonging to a circle. For harmonic radius, this theorem is proved in \cite[Theorem 4]{DP}. Common features of the proof also remain, although we had to extend the technique of p-modules of curves families.
\begin{thm}\label{dis1}
  Let $G$ be either a ring $K(\rho_1,\rho_2)=\{x\in \mathbb{R}^n:\rho_1<|x|<\rho_2\}$ or a cylinder  $Z(\rho_1,\rho_2)=\{[\rho,\theta,x']\in \mathbb{R}^n:\rho_1<\rho<\rho_2\}$, $m\geq 1,$ $0\leq\rho_1\leq\rho_2\leq\infty.$ Then, for any points $a_l$ lying on the circle $O(\rho_0)=\{[\rho,\theta,x']: \rho=\rho_0\}$, $\rho_1<\rho_0<\rho_2$ and any non-overlapping domains $D_l$, $D_l\subset G$, $a_l\in D_l$, $l=0,1,\dots,m-1$, we have
  $$
  \sum_{l=0}^{m-1} \mu_p(R_p(a_l,D_l))\geq \sum_{l=0}^{m-1}\mu_p(R_p(a_l^*,D_l^*))=m\mu_p(R_p(a_0^*,D_0^*)).
  $$
Here
 $$
 a_l^*=\left[\rho_0,\frac{2\pi l}{m}, 0\right] \ \ \text{and} \ \
  D_l^*=G\cap \left\{[\rho,\theta, x']: \frac{\pi(2l-1)}{m}<\theta<\frac{\pi(2l+1)}{m}\right\}.
 $$
\end{thm}

In particular, if $m=2k,$ $\rho_1=0,$ $\rho_2=\infty$ by (\ref{rad2}) we have the inequality \cite{DP}
$$\sum_{l=0}^{2k-1} R_2(a_l, D_l)^{2-n}\geq m \sum_{l=1}^{2k-1}
(-1)^{l+1}|a_{0}^*-a_l^*|^{2-n}.$$
In the case  $m=2,$ $\rho_1=0,$ $\rho_2=\infty$ we obtain  Theorem \ref{lavr1} with the additional condition  that the midpoint of the segment $[a_1, a_2]$ does not belong to the union $D_1\cup D_2.$

Note that Theorem 2 is formulated for a ring or cylinder but the proof presented below can be easily extended to any domain $G$ invariant under arbitrary rotations.

\section{Background results}

Here  we mean by a "curve"  a Borel set $\gamma\in \mathbb{R}^n$  with $s(\gamma)>0,$   where $s(\gamma)$ is one-dimensional Hausdorf measure. We need it, for example, to apply later a dissymmetrization transformation that in a general case breaks a curve.  If we want to emphasize that a curve is understood in a standard sense as a homeomorphic image of a segment or circle  we will call it {\it continuous curve}.  Also, when we say that a curve joins two sets $A$ and $B$ in $G$ we mean that this curve is continuous and has a representation $\gamma: [a,b] \rightarrow \mathbb{R}^n$ such that one of the of the end-points $\gamma(a)$, $\gamma(b)$  belongs to $A$ and the other to $B$, and $\gamma(t)\in G$ for $a<t<b$.

 Let $\Gamma$ be a family of  curves in $\mathbb{R}^n.$  Then the {\it  $p$-modulus} of the curve family    is the following quantity
$$
M_p(\Gamma)=\inf \int_{\mathbb{R}^n}\rho^p dx,
$$
where the inf is taken over all Borel functions $\rho: \mathbb{R}^n\rightarrow[0,\infty]$, such that the inequality $\int_{\gamma}\rho d s\geq 1$ holds for every curve $\gamma\in \Gamma$. Functions $\rho$ that satisfy the mentioned above conditions are called {\it admissible} for the curve family $\Gamma$ and the set of all such functions  is denoted by ${\rm Adm} \Gamma$.  In the case of a family of continuous curves above the notion of the $p$-modulus coincides with the traditional notion of $p$-modulus of families of curves.

 B. Levitski\u{i}  in \cite[Theorem 1]{Lev} showed a connection of the  $p$-harmonic radius  with the $p$-capacity.
Taking into account the equality  the $p$-capacity between the $p$-modulus of a corresponding curve family (see \cite{Shlyk}) we can also define {\it $p$-harmonic radius}  with the help of the following identity
\begin{equation}\label{radmod}
-\mu_p(R_p(a,D))=\lambda_n M_p(t,a,D)^{\frac{1}{1-p}}-\mu_p(t)+o(1), \ \ t\rightarrow 0,
\end{equation}
where $M_p(t,a,D)$ is the $p$-modulus of the curve family $\Gamma(t,a,D)$, which consists of all curves joining the hypersphere $S(a,t)$ and $\partial D$ in $D$, $\lambda_n=(n\omega_n)^{\frac{1}{p-1}}$, $\omega_n$-is the volume of the ball $B(0,1)$.

List here some basic properties of the $p$-modulus (see for example \cite{zbMATH03130373}):
\smallskip

1) If $\Gamma_1\subset \Gamma_2$ then $M_p(\Gamma_1)\leq M_p(\Gamma_2)$.
\smallskip

2) $M_p\left(\cup_{i=1}^{\infty}\Gamma_i\right)\leq \sum_{i=1}^{\infty}M_p(\Gamma_i).$
\smallskip

3) If $\Gamma_2$ is {\it longer} than $\Gamma_1$ which means that each curve $\gamma\in\Gamma_2$ has a subcurve belonging to $\Gamma_1$, then $M_p(\Gamma_1)\geq M_p(\Gamma_2)$.
\smallskip

4) If $\Gamma_1, \Gamma_2,\dots$ are separated and $\Gamma_i$ is longer than $\Gamma$, $i=1,2,\dots$ then $M_p(\Gamma)\geq \sum_{i=1}^{\infty}M_p(\Gamma_i)$. (Curves families $\Gamma_1, \Gamma_2,\dots$ are called  {\it separated} if there exist disjoint Borel sets $E_i$ in $\mathbb{R}^n$ such that if  $\gamma\in\Gamma_i$  then $\int_{\gamma}\chi_i ds=0$, where $\chi_i$ is the characteristic function of $\mathbb{R}^n\setminus E_i$).
\smallskip

5) If $\Gamma_1, \Gamma_2,\dots$ are separated curves families and $\Gamma$ is longer than $\Gamma_i$, $i=1,2,\dots$, then $M_p(\Gamma)^{1/(1-p)}\geq \sum_{i=1}^{\infty}M_p(\Gamma_i)^{1/(1-p)}$.

\begin{lem}\label{szero} Let $L$  be an arbitrary hyperplane and $\Gamma$ consist of curves  $\gamma$   such that the intersection  $\gamma\cap L$ has  positive one-dimensional Hausdorf measure $\left(s(\gamma\cap L)>0\right)$. Then
$$
M_p(\Gamma)=0.
$$
\end{lem}

\begin{proof}
Let $\Gamma_k=\{\gamma\in \Gamma:s(\gamma\cap L)>1/k,\ k\  \text{is\ positive \ integer}\}$. In this case,  $\Gamma=\cup_{k=1}^{\infty}\Gamma_k$. It is easy to verify that the function
$$
\rho(x)=\left\{\begin{array}{l}\  1/k,\ x\in\ L,\\ \
0,\ x\notin\ L,
\end{array}\right.
$$
is admissible for $\Gamma_k$ and $\int_{\mathbb{R}^n}\rho^p dx=0.$ Then by the definition of the $p$-modulus $M_p(\Gamma_k)= 0$. By property 2,  $M_p(\Gamma)=0$.  The lemma is proved.
\end{proof}
Now for  $m\geq 1$, we put
$$
N_k^*=\left\{[\rho,\theta,x']\in \mathbb{R}^n, \ \frac{\pi k}{m}\leq \theta \leq \frac{\pi (k+1)}{m} \right\}, \ k=0,1,\dots, 2m-1,
$$
and
$$
L_k^*=\left\{[\rho, \theta, x']\in \mathbb{R}^n:\theta=\frac{\pi k}{m}\right\}, \ \ k=0,\dots, 2m-1.
$$

We denote by $\Phi$ the group of symmetries in $\mathbb{R}^n$ consisting of the  superpositions
of the reflections in hyperplanes containing $L_l^*$, $l=0,\dots,2m-1$.

\begin{lem}\label{refl}
Let $m\geq 1,$ $\Gamma_0^*$ be a family of curves $\gamma_0^*$ lying in $N_0^*$  and let $\phi_k(x)$  denote the reflection in a hyperplane  containing $L_k^*$. Let  $\gamma_k^*=\phi_k(\gamma_{k-1}^*)$, $k=1,\dots, 2m-1$, and  $\gamma^*=\cup_{k=0}^{2m-1}\gamma_k^*$ be a curve symmetric with respect to the group $\Phi$ and consisting of $2m$ reflections of $\gamma_0^*$, $\Gamma^*$ be the family of curves $\gamma^*$. Then
$$
M_p(\Gamma^*)=(2m)^{1-p}M_p(\Gamma_0^*).
$$
\end{lem}

\begin{proof}
According to lemma \ref{szero} we may assume that the family $\Gamma_0^*$ consists of the curves $\gamma_0^*$ such that
$$
s(\gamma_0^*\cap L_0^*\cap L_1^*)=0,
$$
where  $s$ is one-dimensional Hausdorf  measure. Let
$$
{\rm Int} N_k^*=\left\{[\rho, \theta,x']: \frac{\pi k}{m}<\theta<\frac{\pi(k+1)}{m}\right\}, \ k=0,1,\dots, 2m-1,
$$
be the interior of the angle $N_k^*$.  Denote by $f_k(z)$ the mapping from ${\rm Int} N_0^*$ onto ${\rm Int} N_k^*$ constructed by the following formula
$$
f_0(z)=z\ {\mathrm{and}}\  f_k(z)=\varphi_k (f_{k-1}(z)), \  k=1,\dots, 2m-1.
$$
Let $\rho_0^*\in {\rm Adm} \Gamma_0^*$. Then, for the function
$$
\rho(z)=\left\{\begin{array}{cl}
                 \rho_0^*(f_k^{-1}(z)), & z\in {\rm Int} N_k^*, \\
                 0, & z\in \cup_{k=0}^{2m-1}L_k^*
               \end{array}\right.
$$
we have
$$
\int_{\gamma^*}\rho ds = \sum_{k=0}^{2m-1}\int_{\gamma_k^*}\rho ds = 2m\int_{\gamma_0^*}\rho_0^* ds\geq 2m,
$$
hence $\dfrac{\rho}{2m}\in {\rm Adm} \Gamma^*.$ In view of
$$
\int_{\mathbb{R}^n}\rho^p d\mu = 2m \int_{N_0^*}(\rho_0^*)^p d\mu,
 $$
 we get
$$
(2m)^{1-p}\int_{N_0^*}(\rho_0^*)^p d\mu = \int_{\mathbb{R}^n}\frac{\rho^p}{(2m)^p}d\mu\geq M_p(\Gamma^*).
$$

If we take an infimum then we get
$$
(2m)^{1-p}M_p(\Gamma_0^*)\geq M_p(\Gamma^*).
$$

Now we are going to show the reverse inequality. Let $\rho\in {\rm Adm} \Gamma^*$. We construct a function $\rho_0^*(z)$ by the formula
$$
\rho^*_0(z)=\sum_{k=0}^{2m-1}\rho(f_k(z)), \ \ z\in {\rm Int} N_0^*.
$$
Then
$$
\int_{\gamma_0^*}\rho_0^*ds = \sum_{k=0}^{2m-1}\int_{\gamma_0^*}\rho(f_k(z))ds = \sum_{k=0}^{2m-1}\int_{\gamma_k^*} \rho ds = \int_{\gamma^*}\rho ds\geq 1.
$$
Hence
$$
M_p(\Gamma_0^*)\leq \int _{N_0^*}(\rho_0^*)^p d\mu = \int_{N_0^*}\left(\sum_{k=0}^{2m-1}\rho(f_k(z))\right)^p d\mu
$$
$$
\leq (2m)^{p-1}\sum_{k=0}^{2m-1}\int_{N_0^*}(\rho(f_k(z)))^p d\mu
$$
$$
=(2m)^{p-1}\sum_{k=0}^{2m-1}\int_{N_k^*} \rho^p d\mu =(2m)^{p-1}\int_{\mathbb{R}^n}\rho^p d\mu.
$$
To get the second line here we have applied the following inequality for the  mean values
$$
\left(\frac{t_0^p+t_1^p+\dots +t_{2m-1}^p}{2m}\right)^{\frac{1}{p}}\geq \frac{t_0+t_1+\dots+t_{2m-1}}{2m},
$$
where  $p>1$ and $t_j\geq 0$, $j=0,\dots,2m-1$, are non-negative numbers.

Taking an infimum over $\rho$ we get
$$
M_p(\Gamma_0^*)\leq (2m)^{p-1}M_p(\Gamma^*).
$$
Lemma is proved.
\end{proof}

Now we are going to use the dissymmertrization (see \cite[p. 32] {zbMATH00797804}).  We introduce a symmetric structure $\{P_l\}_{l=1}^N$ in $\overline{\mathbb{R}}^n$ as a collection of closed angles $P_l=\{[\rho, \theta, x'] : \theta_{l_1}\leq \theta \leq \theta_{l_2}\}$, $l=1,\dots, N$, satisfying the conditions:

$aP) \bigcup\limits_{l=1}^N P_l = \overline{\mathbb{R}}^n, \ \ \sum\limits_{l=1}^N (\theta_{l_2}-\theta_{l_1})=2\pi,$

$bP)\  \{\phi(P_l)\}_{l=1}^N=\{P_l\}_{l=1}^N$ for any isometry $\phi\in \Phi$.

Recall that the rotation by an angle $\beta$ is the transformation: $[\rho,\theta,x']\mapsto [\rho, \theta+\beta,x']$.   We call a collection of rotations $\{\alpha_l\}_{l=1}^{N}$ the {\it dissymmetrization} of the symmetric structure $\{P_l\}_{l=1}^{N}$ if the images $S_l=\alpha_l(P_l)$ satisfy the following conditions:

$aS) \bigcup\limits_{l=1}^{N} S_l = \overline{\mathbb{R}}^n$,

$bS)$ for every non-empty intersection $S_l\cap S_p$, $l,p = 1,\dots, N$, there exists

an isometry $\phi \in \Phi$ such that $\phi(\alpha_l^{-1}(S_l\cap S_p))=\alpha_p^{-1}(S_l\cap S_p)$.

Let $A$ be an arbitrary subset of $\overline{\mathbb{R}}^n$. We introduce the notation
$$
{\rm Dis} A = \bigcup_{k=1}^N \alpha_l (A\cap P_l).
$$

We also need the following lemma originally proved by Dubinin in the planar case (see for example \cite[Lemma 4.2]{Dub1}).
\begin{lem}\label{Dub}
Let $m\geq 1,$ $0\leq\theta_0<\theta_1<\ldots<\theta_{m-1}<2\pi, \theta_m=\theta_0+2\pi,$ $\Lambda_l=
\{[\rho,\theta,x']\in\mathbb{R}^n:\theta=\theta_l\}$  and   $\Lambda^*_l=
\{[\rho,\theta,x']\in\mathbb{R}^n:\theta=2\pi l/m\},\ l=0,...,m.$  Then there exists a symmetric structure $\{P_k\}_{k=1}^N$, $N\geq m$, and a dissymmetrization $\{\alpha_k\}_{k=1}^N$ such that ${\rm Dis} \Lambda_l^*=\Lambda_l,$ $l=0,\ldots,m-1$.
\end{lem}
The proof of this lemma practically does not differ from the one in the planar case, so we omit it. In the following lemma we show that dissymmetrization preserves the module of a curve family.

\begin{lem}\label{dis}
If $\Gamma$ is a curve family in $\mathbb{R}^n$ and ${\rm Dis} \Gamma=\{{\rm Dis} \gamma:\gamma\in \Gamma\}$ is the result of the dissymmetrization of the family $\Gamma$ then
$$
M_p(\Gamma)=M_p({\rm Dis}\Gamma).
$$
\end{lem}

\begin{proof}
Let $\{P_l\}_{l=1}^{N}$ be a symmetric structure and $\{\alpha_l\}_{l=1}^N$ be its dissymmetrization, $\alpha_l(P_l)=S_l$.  According to Lemma \ref{szero}, we can assume that the  curves family  $\Gamma$ consists of curves $\gamma$ satisfying the condition $s\left(\cup_{l=1}^N(\partial P_l \cap\gamma)\right)=0.$ Hence ${\rm Dis}\Gamma$ consists of curves ${\rm Dis}\gamma$ satisfying the  similar condition $s\left(\cup_{l=1}^N(\partial S_l \cap {\rm Dis}\gamma)\right)=0.$ Denote by ${\rm Int} P_l$ the set $P_l\setminus \partial P_l$ and  by  ${\rm Int}  S_l$ the set $S_l\setminus \partial S_l$.

If $\rho$ is an admissible function for the curve family $\Gamma$ then
$$
\tilde{\rho}(x)=\left\{\begin{array}{l}\rho(\alpha_l^{-1}x), \ x\in {\rm Int} S_l, \ l=1,\dots,N,\\ 0, \ x\in \cup_{l=1}^N(\partial S_l ),
\end{array}\right.
$$
is admissible for ${\rm Dis}{\Gamma}$. Indeed, for ${\rm Dis}{\gamma}\in {\rm Dis}{\Gamma}$ we get
$$ \int_{{\rm Dis}\gamma }\tilde{\rho} ds=\sum_{l=1}^N\int_{{\rm Dis}{\gamma}\cap{\rm Int}{S}_l}  \tilde \rho ds=\sum_{l=1}^{N}\int_{\gamma\cap {\rm Int}{P}_l}\tilde {\rho}ds\geq 1.
$$
It is easy to see that
$$
\int_{\mathbb{R}^n}\tilde{\rho}^p dx =\sum_{l=1}^{N}\int_{{\rm Int}{S}_l}\tilde {\rho}^p dx = \sum_{l=1}^{N}\int_{{\rm Int}{P}_l}\rho^p dx =\int_{\mathbb{R}^n}\rho^p dx.
$$
By the definition of the $p$-modulus, we get
$$
M_p({\rm Dis}{\Gamma})\leq \int_{\mathbb{R}^n}\rho^p dx.
$$
Taking an infimum over all admissible functions $\rho$ we get $M_p({\rm Dis}{\Gamma})\leq M_p({\Gamma})$.

Similarly, any admissible for the curve family ${\rm Dis}{\Gamma}$ function $\tilde{\rho}(x)$  induces an admissible for ${\Gamma}$ function $\rho(x)$, moreover, again $\int_{\mathbb{R}^n}\tilde{\rho}^p dx=\int_{\mathbb{R}^n}\rho^p dx$. Therefore $M_p({\Gamma})\leq M_p({\rm Dis}{\Gamma})$. Lemma is proved.
\end{proof}

\section{Proofs}

\begin{figure}
\centering
\includegraphics[width=0.9\textwidth]{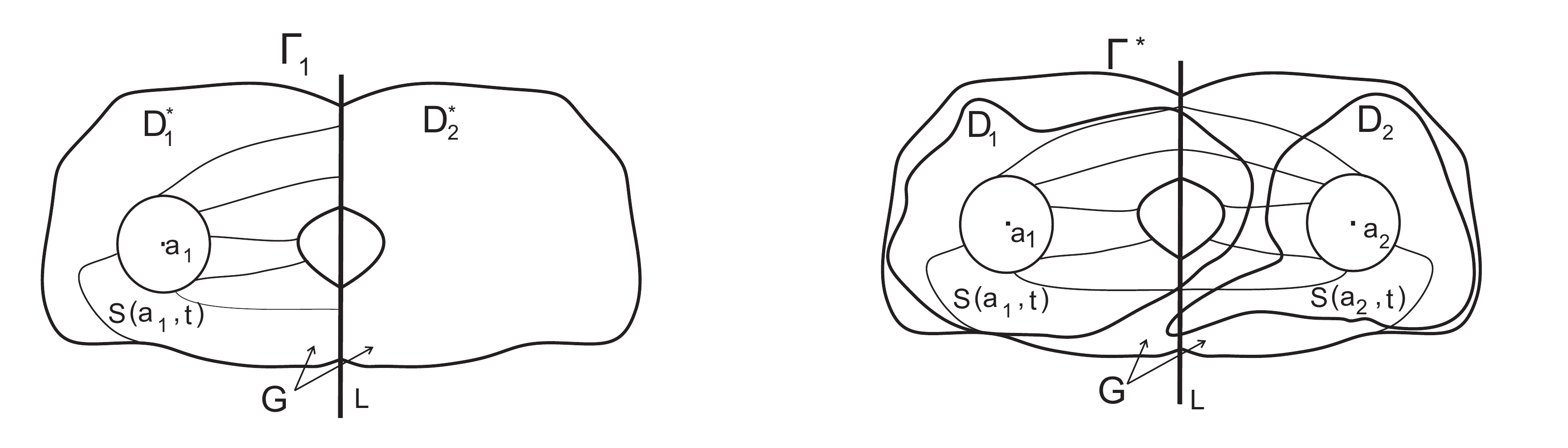}
 \caption{Families $\Gamma_1$ and $\Gamma^*$}
\end{figure}

\begin{proof}[\bf Proof of theorem 1.]  As above, $\Gamma(t,a,D)$ denotes  the family of all curves joining $S(a,t)$ and $\partial D$ in $D$. Let $\Gamma_1=\Gamma(t,a_1,D_1^*)$, $\Gamma^*=\{\gamma\cup\gamma^*: \gamma\in \Gamma_1, \gamma^* \ \text{is a reflection of}\  \gamma \ \text{in} \ L\}$ (see Figure 2). Consider a curve $\tilde{\gamma} \in \Gamma^*$ generated by a curve $\gamma\in \Gamma_1$. Assume that  ${\gamma}$ joins a point $a\in S(a_1,t)$ with a point $b\in \partial D_1^*.$ There are only two possibilities: either $b\in G\cap L$ or $b\in \partial G.$   In the first case $\tilde{\gamma}$ is a continuous curve that joins $S(a_1,t)$ with $S(a_2,t)$ in the set $G$. Hence, there is a subcurve of $\tilde{\gamma}$ that joins $S(a_1,t)$ and $\partial D_1$. In the second case  (when $\gamma$ joins $S(a_1,t)$ and $\partial G$), we have either the curve $\gamma$ itself joins $S(a_1,t)$ with $\partial D_1$ or  there is a point of intersection $\gamma\cap\partial D_1$. In both these cases there exists a subcurve of $\gamma\subset\tilde{\gamma}$ that joins $S(a_1,t)$ and $\partial D_1$. It means that in any case there exists a subcurve of $\tilde{\gamma}$ joining $S(a_1,t)$ and $\partial D_1$. Therefore  $\Gamma^*$ is longer than $\Gamma(t,a_1,D_1)$. Similarly $\Gamma^*$ is longer than $\Gamma(t,a_2,D_2)$. By Lemma \ref{refl}, $M_p(\Gamma^*)=2^{1-p}M_p(\Gamma_1)$ ($m=1$), and by property 5,
$$
M_p(\Gamma^*)^{\frac{1}{1-p}}\geq M_p(\Gamma(t,a_1,D_1))^{\frac{1}{1-p}}+M_p(\Gamma(t,a_2,D_2))^{\frac{1}{1-p}}
$$
or, equivalently,
$$
2M_p(\Gamma_1)^{\frac{1}{1-p}}\geq M_p(\Gamma(t,a_1,D_1))^{\frac{1}{1-p}}+M_p(\Gamma(t,a_2,D_2))^{\frac{1}{1-p}}.
$$
Multiplying  by $\lambda_n$ and subtract $2\mu_p(t)$, we get
$$
2(\lambda_nM_p(\Gamma_1)-\mu_p(t))\geq \sum_{i=1}^2 (\lambda_nM_p(t,a_i,D_i)-\mu_p(t)),
$$
taking a limit as $t\rightarrow 0$,  we obtain by (\ref{radmod})
$$
-2\mu_p(R_p(a_1,D_1^*))\geq -\mu_p(R_p(a_1,D_1))-\mu_p(R_p(a_2,D_2)).
$$
The theorem is proved.
\end{proof}

\begin{proof}[\bf Proof of theorem 2.]
\begin{figure}
\centering
\includegraphics[width=1\textwidth]{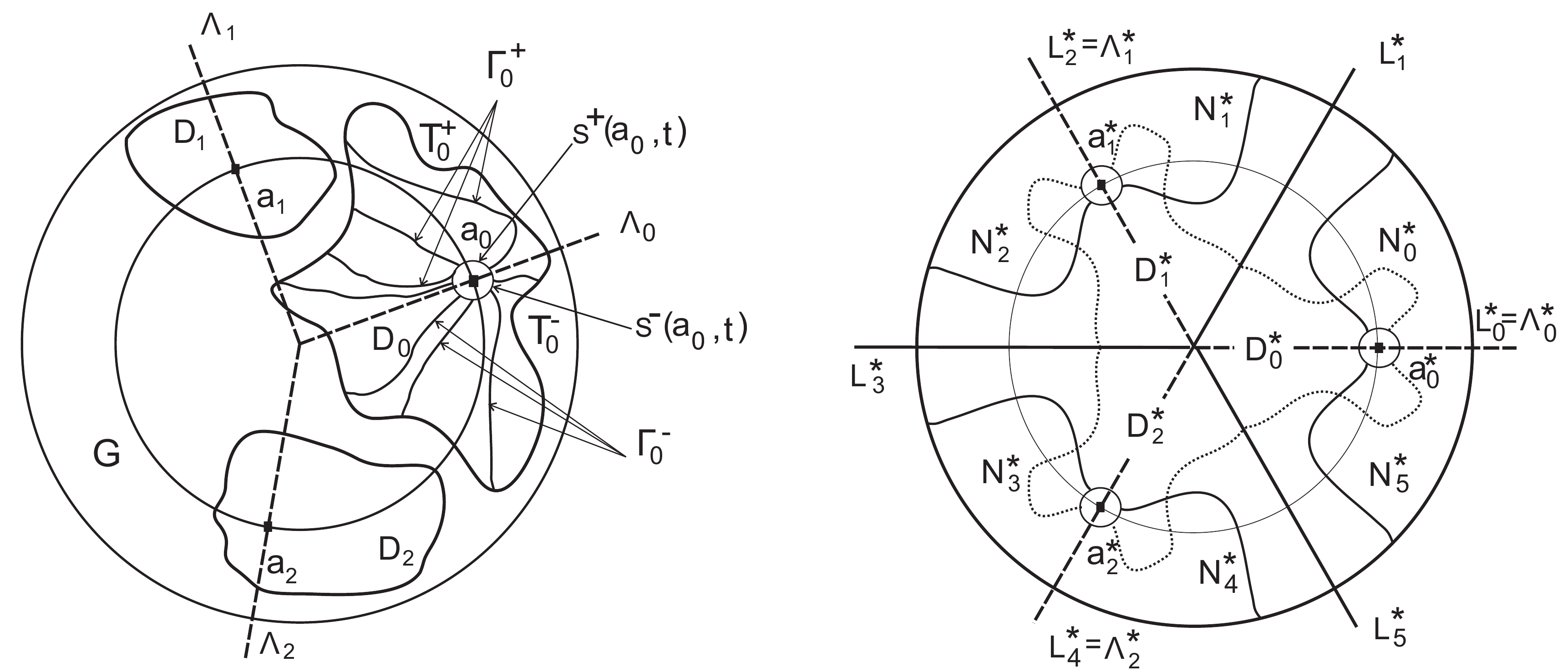}
\label{fig:fam1} \caption{Arbitrary and extremal configurations ($m=3$)}
\end{figure}
 Introduce the following notation
 $$
 a_l=[\rho_0,\theta_l,0], \ 0\leq \theta_l<2\pi,\
 \Lambda_l=\{[\rho,\theta, x']: \theta=\theta_l\}, \  l=0,\dots, m-1.
 $$
Without loss of generality we will assume that $\theta_0<\theta_1<\theta_2<\dots<\theta_{m-1}$. Let
$$
T_l^+=\{[\rho,\theta,x']:\theta_l\leq\theta\leq\theta_{l+1}\},\
T_l^-=\{[\rho,\theta,x']:\theta_{l-1}\leq\theta\leq\theta_l\}
$$
for $l=1,\dots, m-1$,  and
$$
T_0^+=\{[\rho,\theta,x']:\theta_0\leq\theta \leq\theta_1\}, \
T_0^-=\{[\rho,\theta,x']:\theta_{m-1}\leq\theta\leq\theta_0+2\pi\}.
$$
$$
S^+(a_l,t)=S(a_l,t)\cap T_l^+, \ \ S^-(a_l,t)=S(a_l,t)\cap T_l^-,
$$
$l=0,\dots,m-1$.
Moreover, let
$$
\Lambda^*_l=
\{[\rho,\theta,x']\in\mathbb{R}^n:\theta=2\pi l/m\},\ l=0,...,m,
$$
and, as above,
$$
L_k^*=\left\{[\rho, \theta, x']\in \mathbb{R}^n:\theta=\frac{\pi k}{m}\right\}, \ \ k=0,\dots, 2m-1,
$$
$$
N_k^*=\left\{[\rho,\theta,x']\in \mathbb{R}^n: \ \frac{\pi k}{m}\leq \theta \leq\frac{\pi (k+1)}{m} \right\}, \ k=0,1,\dots, 2m-1.
$$

Arbitrary and extremal configurations are depicted on Figure~3 for $m=3$. Denote by $\Gamma_0^*$ the family of all continuous curves $\gamma_0^*$ from $\Gamma(t,a_0^*,D_0^*)$ that join $S(a_0^*,t)\cap N_0^*$ with the boundary $\partial D_0^*$ in the set $N_0^*$ and $s(\gamma_0^*\cap L_0^*)=0$. Let $\Gamma_l^+$ ($\Gamma_l^-$) be families of all continuous curves $\gamma_l^+$ ($\gamma_l^-$) from  $\Gamma(t,a_l,D_l)$ joining $S^+(a_l,t)$ ($ S^-(a_l,t)$) with $\partial D_l$ in $T_l^+$ ($T_l^-$) and $s(\gamma_l^+\cap \Lambda_l)=0$, ($s(\gamma_l^-\cap \Lambda_l)=0$). By  lemma~\ref{Dub}, if $t$ is  small enough, there is a dissymmetrization  "moving" $\Lambda_l^*$ ($a_l^*\in \Lambda_{l}^*$) to $\Lambda_l$ such that ${\rm Dis}\, S(a_l^*,t)=S(a_l,t),$ $l=0,\dots m-1$. Such dissymmetrization is depicted on Figure~4.

As in Lemma \ref{refl}, we construct a family $\Gamma^*$ by $2m$ reflections of each curve $\gamma_0^*\in \Gamma_0^*$.
Show that ${\rm Dis} \Gamma^*$ is longer than $\Gamma_0^+$.

Let $\gamma^*\in \Gamma^*$ be a curve generated by a continuous curve $\gamma_0^*\in \Gamma_0^*$. If $\gamma_0^*$ joins $S(a_0^*,t)$ with $\partial G$ we supply it  by a curve $\gamma_1\subset \partial G$ such that $\gamma_0^*\cup \gamma_1$ joins $S(a_0^*,t)$ and $L_1^*$ in $P_0^*$. Note, that if $\gamma_0^*$ joins $S(a_0^*,t)$ with $L_1^*$ then this construction is   superfluous. Then we connect a point $z\in \gamma_0^*\cap S(a_0^*,t)$ with $L_0^*$ by a curve $\gamma_2\subset S(a_0^*,t).$

 Let $\tilde\gamma_0^*=\gamma_1\cup\gamma^*_0\cup\gamma_2.$ Similarly as in Lemma \ref{refl}, we construct $\tilde{\gamma}^*$ for the curve $\tilde{\gamma}_0^*$. The  continuous curve $\tilde{\gamma}^*$ joins successively the hyperspheres $S(a_l^*,t)$, $l=0,\dots,m-1$ so that the upper hemihypersphere $S^+(a_l^*,t)$ is connected with the lower hemihypersphere  $S^-(a^*_{l+1},t)$ ($a_{m}^*=a_0^*$). Since the  symmetry of $\tilde{\gamma}^*$ taking into account the property $bS$ of dissymmetrization we get that the curve ${\rm Dis} \tilde{\gamma}^*$ contains a continuous curve joining the hemihyperspheres $S^+(a_0,t)$ and $S^-(a_1,t)$ in $T_0^+$. By the conditions of the theorem,  $D_0$ contains the ball $B(a_0,t)$, $D_1$ contains $B(a_1,t)$ and $D_0\cap D_1=\emptyset$. Therefore there is a subcurve $\gamma\subset {\rm Dis} \tilde{\gamma}^*$ joining $S^+(a_0,t)$ and $\partial D_0$ in the set $T_0^+$. Let $b_0$ be a point from $S^+(a_0,t)\cap \gamma$ and $b_1\in \gamma\cap \partial D_0$ be the  closest point  to  $b_0$  of $\partial D_0$ on $\gamma$.    Then the part of $\gamma$ between $b_0$ and $b_1$  contains a continuous curve joining $\partial D_0$ and $S^+(a_0,t)$ in $D_0\setminus \overline{B(a_0,t)}$. This curve is a subcurve of $\rm{Dis}\gamma^*$ and belongs to $\Gamma_0^+.$ Therefore, ${\rm Dis} \Gamma^*$ is longer than $\Gamma_0^+.$

\begin{figure}
\centering
\includegraphics[width=1\textwidth]{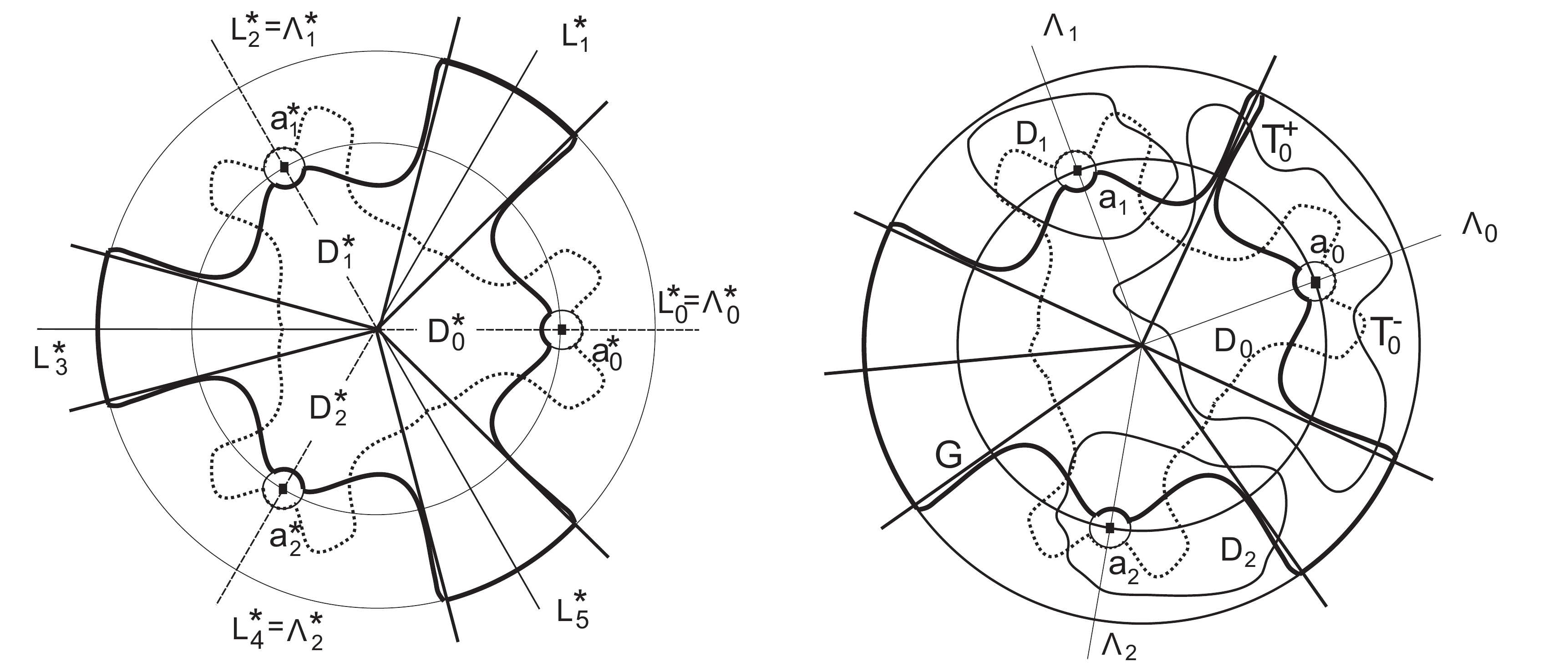}
\label{fig:diss} \caption{Dissymmetrization}
\end{figure}

Similarly, ${\rm Dis} \Gamma^*$ is longer than $\Gamma_l^+$ and $\Gamma_l^-$ for all $l=0,\dots, m-1$. The families $\Gamma_l^+$ and $\Gamma_l^-$ are separated. By  property 5, Lemma \ref{dis} and Lemma \ref{refl},  we get the following inequality
$$
\left((2m)^{1-p}M_p(\Gamma_0^*)\right)^{\frac{1}{1-p}}=M_p(\Gamma^*)^{\frac{1}{1-p}}= M_p({\rm Dis} \Gamma^*)^{\frac{1}{1-p}}
$$
$$
\geq \sum_{l=0}^{m-1}\left(M_p(\Gamma_l^+)^{\frac{1}{1-p}}+M_p(\Gamma_l^-)^{\frac{1}{1-p}}\right).
$$
On the other hand, $\Gamma_l^+$ and $\Gamma_l^-$ both are longer than $\Gamma(t,x_l,D_l)$. By  property~4,
$$
M_p(\Gamma(t,a_l,D_l))\geq M_p(\Gamma_l^+)+M_p(\Gamma_l^-).
$$
If $d<0,\ u\geq 0,\ v\geq 0,$ then the following inequality for means
$$
\left(\frac{u^d+v^d}{2}\right)^{\frac{1}{d}}\leq \frac{u+v}{2}
$$
holds, or equivalently
$$
u^d+v^d\geq 2^{1-d}(u+v)^d.
$$
Applying this inequality for $d=\frac{1}{1-p}$ we get
$$
2mM_p(\Gamma_0^*)^{\frac{1}{1-p}}\geq \sum_{l=0}^{m-1}\left(M_p(\Gamma_l^+)+M_p(\Gamma_l^-)\right)^{\frac{1}{1-p}}2^{1-\frac{1}{1-p}}\geq
$$
$$
2^{1-\frac{1}{1-p}}\sum_{l=0}^{m-1}M_p(\Gamma(t, a_l,D_l))^{\frac{1}{1-p}}.
$$
It can be rewritten in the following form
$$
m(2M_p(\Gamma_0^*))^{\frac{1}{1-p}}\geq\sum_{l=0}^{m-1}M_p(\Gamma(t,a_l,D_l))^{\frac{1}{1-p}}.
$$
By the principle of symmetry \cite[Lemma 5.20, p. 55]{Vu} taking into account Lemma~\ref{szero}, we get
$$
2M_p(\Gamma_0^*)=M_p(\Gamma(t,a_0^*,D_0^*)).
$$
Therefore,
$$
mM_p(\Gamma(t, a_0^*,D_0^*))^{\frac{1}{1-p}}\geq \sum_{l=0}^{m-1}M_p(\Gamma(t,a_l,D_l)).
$$
We multiply this inequality by $\lambda_n$ and subtract $m \mu_p(t)$
$$
m\left(\lambda_n M_p(\Gamma_t(a_0^*,D_0^*)^{\frac{1}{1-p}})-\mu_p(t)\right)\geq \sum_{l=0}^{m-1} \left(\lambda_n M_p(\Gamma(t,a_l,D_l))^{\frac{1}{1-p}}-\mu_p(t)\right).
$$
Taking a limit as $t\rightarrow 0$ we obtain
$$
-m \mu_p(R_p(a_0^*,D_0^*))\geq -\sum_{l=0}^{m-1}\mu_p(R_p(a_l,D_l)).
$$
 Theorem is proved.
 \end{proof}

 This work has been supported by the Russian Science Foundation under project 14-11-00022

Sergei Kalmykov
\\
Department of Mathematics, Shanghai Jiao Tong University,
800 Dongchuan RD, Shanghai, 200240, China
\\
email address: \href{mailto:sergeykalmykov@inbox.ru}{sergeykalmykov@inbox.ru}

\medskip{}

Elena Prilepkina
\\
Institute of Applied Mathematics, FEBRAS, 7 Radio Street, Vladivostok,
690041, Russia
\\
and
\\
Far Eastern Federal University, 8 Sukhanova Street, Vladivostok, 690950,
Russia
\\
email address: \href{mailto:pril-elena@yandex.ru}{pril-elena@yandex.ru}


\begin{thebibliography}{99}
\bibitem{Lev} B. Levitski\u{i}, Reduced p-modulus and the interior p-harmonic radius, Dokl. Akad. Nauk SSSR {\bf 316 (4)} (1991) 812--815 (in Russian); translation in: Soviet Math. Dokl. {\bf 43(1)} (1991) 189--192.

\bibitem{Ban}  C.Bandle and M.Flucher.
 Harmonic radius and concentration of energy, hyperbolic radius and Liouvilles equations $\Delta U=0$ and $\Delta U=U^{\frac{n+2}{n-2}}$.  SIAM Review {\bf 38(2)}  (1996) 191--238.

\bibitem{Dub1} V.N. Dubinin. Condenser capacities and symmetrization in geometric function theory. Birkhäuser Basel, 2014.

\bibitem{DP}  V.N. Dubinin and E.G. Prilepkina, Extremal decomposition of spatial domains, Journal of Mathematical Sciences, {\bf 105(4)} (2001) 2180--2189.

\bibitem{GKP} K.A. Gulyaeva, S.I. Kalmykov, E.G. Prilepkina,  Extremal decomposition problems in the Euclidean space. International Journal of Mathematical Analysis {\bf 9(56)} (2015) 2763--2773.


\bibitem{W} W. Wang, N-Capacity, N-harmonic radius and N-harmonic transplantation. J. Math. Anal. Appl. {\bf 327(1)} (2007) 155--174.

\bibitem{Kuz} G.V. Kuz’mina, The method of extremal metric in extremal decomposition problems with free parameters. Journal of Mathematical Sciences {\bf 129(3)} (2005)  3843--3851.

\bibitem{Emel} E.G. Emelyanov, On the problem of maximizing the product of powers of conformal radii nonoverlapping domains. Journal of Mathematical Sciences
{\bf 122(6)} (2004)  3641--3647.


\bibitem{Vas} A.Vasil'ev. Moduli of families of curves for conformal and quasiconformal mappings. Lecture Notes in Mathematics, vol. 1788, Springer-Verlag, Berlin- New York, 2002, 212 pp.

\bibitem{VP} Ch.~Pommerenke, A.~Vasil’ev. Angular derivatives of bounded univalent functions and extremal partitions of the unit disk.  Pacific Journal of Mathematics  {\bf 206(2)} (2002) 425--450.

\bibitem{Sol} A.Yu. Solynin. Decomposition into nonoverlapping domains and extremal properties of univalent functions. Journal of Mathematical Sciences {\bf 83(6)} (1997) 779--794.



\bibitem{HKM} J. Heinonen, T. Kilpeläinen, O. Martio, Nonlinear Potential Theory of Degenerate Elliptic Equations, Oxford Math. Monogr., Oxford Univ. Press, New York, 1993.

\bibitem{KV} S. Kichenassamy, L. Veron, Singular solutions of the p-Laplace equation, Math. Ann. {\bf 275} (1986) 599--615; Erratum: Math. Ann. {\bf 277(2)} (1987) 352.


\bibitem{Shlyk} V.A. Shlyk. The equality between $p$-capacity and $p$-modulus. Siberian Mathematical Journal  {\bf 34(6)} (1993)  1196--1200.

\bibitem{zbMATH03130373}
Bent Fuglede. Extremal length and functional completion. Acta Mathematica {\bf 98(1)} (1957) 171--219.



\bibitem{zbMATH00797804} V.N. Dubinin.  Symmetrization in the geometric theory of functions of a complex variable. Russian Mathematical Surveys {\bf 49(1)} (1994) 1-79.

\bibitem{Vu} M. Vuorinen.  Conformal geometry and quasiregular mappings. Lecture Notes in Mathematics. Springer-Verlag. 1988.
\end{thebibliography}
\end{document}